\documentclass[12pt]{amsart}  

\usepackage[latin1]{inputenc}
\usepackage{amsmath} 
\usepackage{amsfonts}
\usepackage{amssymb}
\usepackage{stmaryrd}
\usepackage{latexsym} 
\usepackage{graphicx}
\usepackage{subfigure}
\usepackage{color}
\usepackage{hyperref}
\usepackage{verbatim}
\usepackage[all]{xy}
\usepackage{graphics}
\usepackage{pdfsync}
\usepackage{xcolor}
\usepackage{tikz}
\usepackage{bm}
\usetikzlibrary{arrows.meta}
\usetikzlibrary{shapes.geometric}

\theoremstyle{definition}
\newtheorem{theorem}{Theorem}[section]
\newtheorem{proposition}[theorem]{Proposition}

\newtheorem{lemma}[theorem]{Lemma}
\newtheorem{corollary}[theorem]{Corollary}

\newtheorem{conjecture}[theorem]{Conjecture}
\newtheorem{definition}[theorem]{Definition}
\newtheorem{example}[theorem]{Example}

\newtheorem{definition/theorem}[theorem]{Definition/Theorem}

\theoremstyle{remark}
\newtheorem{remark}[theorem]{Remark}

\numberwithin{equation}{section}
\newcommand{\ka}{\kappa}
\newcommand{\la}{\lambda}
\newcommand{\bx}{\bm x}

\newcommand{\dom}{\unrhd}

\DeclareMathOperator{\Asc}{Asc}
\DeclareMathOperator{\asc}{asc}
\DeclareMathOperator{\Inv}{Inv}
\DeclareMathOperator{\inv}{inv}
\DeclareMathOperator{\Fix}{Fix}
\DeclareMathOperator{\HT}{HT}
\DeclareMathOperator{\sgn}{sgn}

\newcommand{\al}{\alpha}

\newlength\cellsize 
\savebox2{%
\begin{picture}(15,15)
\put(0,0){\line(1,0){15}}
\put(0,0){\line(0,1){15}}
\put(15,0){\line(0,1){15}}
\put(0,15){\line(1,0){15}}
\end{picture}}
\newcommand\cellify[1]{\def\thearg{#1}\def\nothing{}%
\ifx\thearg\nothing
\vrule width0pt height\cellsize depth0pt\else
\hbox to 0pt{\usebox2\hss}\fi%
\vbox to 15\unitlength{
\vss
\hbox to 15\unitlength{\hss$#1$\hss}
\vss}}
\newcommand\tableau[1]{\vtop{\let\\=\cr
\setlength\baselineskip{-16000pt}
\setlength\lineskiplimit{16000pt}
\setlength\lineskip{0pt}
\halign{&\cellify{##}\cr#1\crcr}}}
\savebox3{%
\begin{picture}(15,15)
\put(0,0){\line(1,0){15}}
\put(0,0){\line(0,1){15}}
\put(15,0){\line(0,1){15}}
\put(0,15){\line(1,0){15}}
\end{picture}}
\newcommand\expath[1]{%
\hbox to 0pt{\usebox3\hss}%
\vbox to 15\unitlength{
\vss
\hbox to 15\unitlength{\hss$#1$\hss}
\vss}}
\newcommand\bas[1]{\omit \vbox to \cellsize{ \vss \hbox to \cellsize{\hss$#1$\hss} \vss}}

\usepackage[backend=bibtex]{biblatex}
\begin{document}

\title[Chromatic symmetric functions and change of basis]{Chromatic symmetric functions and change of basis}

\author{Bruce E. Sagan}
\address{Department of Mathematics, Michigan State University, East Lansing, MI 48824}
\email{bsagan@msu.edu}

\author{Foster Tom}
\address{Department of Mathematics, Massachusetts Institute of Technology, Cambridge, MA 02139}
\email{ftom@mit.edu}


\begin{abstract}
We prove necessary conditions for certain elementary symmetric functions, $e_\la$, to appear with nonzero coefficient in Stanley's chromatic symmetric function as well as in the generalization considered by Shareshian and Wachs. We do this by first considering the expansion in the monomial or Schur basis and then performing a basis change.
Using the former, we make a connection with two fundamental graph theory invariants, the independence and clique numbers.
This allows us to prove nonnegativity of three-column coefficients for all natural unit interval graphs. The Schur basis permits us to give 
a new interpretation of the coefficient of $e_n$ in terms of tableaux.  We are also able to give an explicit formula for that coefficient.
\end{abstract}

\maketitle

\section{Introduction}\label{section:introduction}
Let $G=(V,E)$ be a graph with vertex set $V=\{v_1,\ldots,v_n\}$ and edge set $E$.
Also let $\bm x=\{x_1,x_2,\ldots\}$ be a set of commuting variables indexed by the positive integers $\mathbb N$.
In his groundbreaking 1995 work, Stanley defined the \emph{chromatic symmetric function} \cite[Definition 2.1]{chromsym}
$$
X_G(\bm x)=\sum_\kappa x_{\kappa(v_1)}\cdots x_{\kappa(v_n)},
$$
where the sum is over all \emph{proper colourings} of $G$, which are functions $\kappa:V\to\mathbb N$ such that whenever $ij\in E$, we have $\kappa(i)\neq\kappa(j)$. Consider the expansion 
\begin{equation}
\label{c_la}
X_G(\bm x)=\sum_\lambda c_\lambda e_\lambda
\end{equation}
in the basis of \emph{elementary symmetric functions}. 
We reserve the notation $c_\la$ for these coefficients.
We say that $X_G(\bx)$ is \emph{$e$-positive} if all the $c_\la$ are nonnegative, that is, every $e_\la$ which appears does so with positive coefficient.
Stanley and Stembridge  \cite[Conjecture 5.1]{chromsym}, \cite[Conjecture 5.5]{stanstem} made the following conjecture which has become one of the driving forces behind the study of $X_G(\bx)$.
\begin{conjecture}[$(\bm 3+\bm 1)$-free Conjecture]
\label{SSconj}
 If $G$ is the incomparability graph of a $(\bm 3+\bm 1)$-free poset, 
 then $X_G(\bm x)$ is $e$-positive.
\end{conjecture}
Guay-Paquet \cite[Theorem 5.1]{stanstemreduction} showed that it suffices to prove $e$-positivity whenever $G$ is the incomparability graph of a 
poset which is both $(\bm 3+\bm 1)$- and $(\bm 2+\bm 2)$-free.  These are called {\em unit interval graphs}. Shareshian and Wachs defined a $q$-analogue of $X_G(\bm x)$~\cite[Definition 1.2]{chromposquasi} as follows.
Suppose $G$ has vertex set $V=[n]:=\{1,\ldots,n\}$.  Now a proper coloring 
$\ka:[n]\rightarrow\mathbb N$ has {\em ascent number}
$$
\asc\ka = \#\{ ij \in E \mid i<j \text{ and } \ka(i)<\ka(j)\},
$$
where we will use a hash tag or a pair of vertical bars to denote cardinality.  Now define the \emph{chromatic quasisymmetric function} to be
$$
X_G(\bm x;q)=\sum_\kappa q^{\text{asc}(\kappa)}x_{\kappa(1)}\cdots x_{\kappa(n)},
$$
summed over proper $\ka:[n]\rightarrow\mathbb N$.
To generalize Conjecture~\ref{SSconj} to this setting, it suffices to consider unit interval graphs with a particular labeling.  Say that $G=([n],E)$ is a {\em natural unit interval graph}  if for all $1\leq i<j<k\leq n$ we have
\begin{equation}\label{eq:uicondition}
ik\in E \text{ implies } ij\in E \text{ and } jk\in E.
\end{equation}
Although $X_G(\bm x;q)$ is not generally symmetric in the $\bm x$ variables, Shareshian and Wachs \cite[Theorem 4.5]{chromposquasi} showed that $X_G(\bm x;q)$ is symmetric whenever $G$ is a natural unit interval graph. 
So we have an elementary symmetric function expansion
\begin{equation}
\label{c_la(q)}
X_G(\bm x;q)=\sum_\lambda c_\lambda(q)e_\lambda
\end{equation}
where the $c_\la(q)$ are polynomials in $q$.  
Again, the notation $c_\la(q)$ will always refer to these coefficients.
Call the expansion {\em $e$-positive} if the coefficients in each $c_\la(q)$ are nonnegative.
This leads to the following conjecture.
\begin{conjecture}[Shareshian-Wachs]
\label{SWconj}
 If $G$ is a natural unit interval graph
 then $X_G(\bm x;q)$ is $e$-positive.
\end{conjecture}
Note that when we set $q=1$ we recover Conjecture~\ref{SSconj}, so $c_\lambda(1)=c_\lambda$.
Several authors proved $e$-positivity for particular classes of natural unit interval graphs $G$. Harada and Precup \cite[Theorem 6.1]{stanstemhess} proved that $X_G(\bm x;q)$ is $e$-positive whenever $G$ has independence number $2$ using cohomology of abelian Hessenberg varieties and Cho and Huh \cite[Theorem 3.3]{chromstoe} gave a purely combinatorial proof. Cho and Hong \cite[Theorem 1.8]{chrombounce3} proved that $X_G(\bm x)$ is $e$-positive whenever $G$ has independence number $3$. Dahlberg \cite[Corollary 5.4]{trilad} proved that $X_G(\bm x)$ is $e$-positive whenever $G$ has clique number $3$. Gebhard and Sagan \cite[Corollary 7.7]{chromnsym} proved that $X_G(\bm x)$ is $e$-positive whenever $G$ is formed by joining a sequence of cliques at single vertices and Tom \cite[Corollary 4.20]{qforesttriples} proved that $X_G(\bm x;q)$ is $e$-positive for such graphs by finding an explicit formula. 

One can alternatively take a dual approach to the Stanley--Stembridge conjecture by fixing a particular partition $\lambda$ and proving nonnegativity of $c_\lambda(q)$ or $c_\lambda$ for all natural unit interval graphs $G$. Hwang \cite[Theorem 5.13]{chromtwoparthook} proved that $c_\lambda(q)$ is nonnegative whenever $\lambda$ is a hook. Abreu and Nigro \cite[Corollary 1.10]{chromhesssplitting} and independently Rok and Szenes \cite[Theorem 4.1]{chromeschers} proved that $c_\lambda$ is nonnegative whenever $\lambda$ has exactly two rows. Clearman, Hyatt, Shelton and Skandera \cite[Theorem 10.3]{chromhecke} proved that $c_\lambda(q)$ is nonnegative whenever $\lambda$ has exactly two columns, using a connection to characters of Hecke algebras.

The rest of this paper is structured as follows.  In the next section, we will collect the results we have obtained about the $e$-expansions of $X_G(\bm x)$ and $X_G(\bm x;q)$ by passing through the monomial symmetric function basis.  In particular, we are able to relate the appearance of an $e_\la$ in these symmetric functions to two fundamental graphical invariants called the independence and clique numbers.  As a consequence, we are able to show that for certain partitions $\la$,
all graphs considered in Conjectures~\ref{SSconj} or Conjecture~\ref{SWconj} contain $e_\la$ with nonnegative coefficient.  In Section~\ref{sec:Schur} we pass through the Schur basis to obtain a new combinatorial interpretation of the coefficient of $e_n$ in $X_G(\bm x;q)$ in terms of tableaux.  
This interpretation allows us to give an explicit formula for this coefficient in terms of $q$-integers.
We also provide a bijection between these tableaux and the acyclic orientations of $G$ known to be counted by this coefficient.
We end with a section giving further directions for research and a log-concavity conjecture.
Throughout, we will assume that $G$ is a graph with vertices $[n]$ and edges $E$ unless otherwise stated.

\section{The monomial basis}\label{sec:alphaomega}

Our main tool in this section is a result we call the Alpha-Omega Lemma (Lemma~\ref{lem:alphaomega}) which 
will connect the coefficients in~\eqref{c_la} and~\eqref{c_la(q)} to two classical graph invariants, namely the independence and clique numbers of $G$.
This will allow us to make progress on Conjectures~\ref{SSconj} and~\ref{SWconj}.
We will also prove a result about the divisibility of these coefficients by a product of certain factorials. Background on partitions and symmetric functions can be found in the books of Sagan~\cite{symgroup} or Stanley \cite{enum2}.

\begin{definition}
An \emph{independent set} in $G$ is a subset of vertices $I\subseteq [n]$ where every pair is not joined by an edge. The \emph{independence number of $G$} is 
$$
\alpha(G) = \text{ the size of the largest independent set in $G$.}
$$
A \emph{clique} in $G$ is a subset of vertices $C\subseteq [n]$ where every pair is joined by an edge. 
The \emph{clique number of $G$} is
$$
\omega(G) = \text{ the size of the largest clique in $G$.}
$$
\end{definition}

We denote by $\lambda'$ the conjugate of a partition $\lambda$. Note that $\lambda'_1$ is the number of parts of $\lambda$. We can now prove one of the main tools of this section.

\begin{lemma}[Alpha-Omega Lemma]
\label{lem:alphaomega} Let $G$ be a graph and let $X_G(\bm x)=\sum_\lambda c_\lambda e_\lambda$ be its chromatic symmetric function. If $c_\lambda\neq 0$, then we must have

\begin{enumerate}
    \item $\alpha(G) \geq \lambda_1'$, and
    \item $\omega(G)\leq \lambda_1$.
\end{enumerate}

Moreover, suppose that $G$ is a graph for which the chromatic quasisymmetric function is in fact symmetric, and let $X_G(\bm x;q)=\sum_\lambda c_\lambda(q)e_\lambda$. If $c_\lambda(q)\neq 0$, then again both inequalities hold.
\end{lemma}

\begin{proof}
First consider the expansion 
$$
X_G(\bm x)=\sum_\mu a_\mu m_\mu
$$
in the basis of monomial symmetric functions. If $a_\mu\neq 0$, then there must exist  a proper colouring with $\mu_1$ $1$'s, $\mu_2$ $2$'s and so on. Because the set of vertices coloured $1$ must form an independent set, we must have $\alpha(G)\geq\mu_1$. And because the vertices of a clique must be coloured with distinct colours, we must have $\omega(G)\leq\mu'_1$. 

Now the change of basis from monomial to elementary symmetric functions has the form
\begin{equation}
m_\mu=\sum_\lambda A_{\lambda,\mu}e_\lambda,
\end{equation}
where $A_{\lambda,\mu}=0$ unless the dominance relation $\mu'\unlhd\lambda$ holds. Therefore, if $c_\lambda\neq 0$, then there must be some $\mu$ with $\mu'\unlhd\lambda$ and $a_\mu\neq 0$, and so
\begin{equation}
\alpha(G)\geq\mu_1\geq\lambda'_1\text{ and }\omega(G)\leq\mu'_1\leq\lambda_1.
\end{equation}

Finally, if $X_G(\bm x;q)$ is symmetric and $X_G(\bm x;q)=\sum_\mu a_\mu(q)m_\mu$, then $a_\mu(q)$ is now the sum of $q^{\asc \kappa}$ over proper colourings with $\mu_1$ $1$'s, $\mu_2$ $2$'s, and so on. In particular, if $a_\mu(q)\neq 0$, then there must exist such a proper colouring. The rest of the argument proceeds as before. 
\end{proof}

\begin{remark}
This argument also holds with $\omega(G)$ replaced by the \emph{chromatic number} $\chi(G)$, the minimum number of colours needed in a proper colouring of $G$. This is a stronger result in general because $\chi(G)\geq\omega(G)$. However, we are primarily interested in natural unit interval graphs, and we will see in Proposition \ref{prop:domsmallest} that $\chi(G)=\omega(G)$ for such graphs.
\end{remark}

\begin{example}
The bowtie graph $G$ in Figure \ref{fig:alphaomegaexample} has $\alpha(G)=2$ and $\omega(G)=3$, so all terms $e_\lambda$ appearing in $X_G(\bm x)$ must have first part at least $3$ and length at most $2$. The claw graph $H$ in Figure \ref{fig:alphaomegaexample} has $\alpha(H)=3$ and $\omega(H)=2$, so all terms $e_\lambda$ appearing in $X_H(\bm x)$ must have first part at least $2$ and length at most $3$. Note that Lemma \ref{lem:alphaomega} holds even though $X_H(\bm x)$ is not $e$-positive. Because $G$ is a natural unit interval graph, $X_G(\bm x;q)$ is symmetric and the same conditions are required for $e_\lambda$ to appear. We can verify by direct computation that
\begin{equation}
X_G(\bm x;q)=q^2(1+q)^2e_{32}+q(1+q)^2(1+q+q^2)e_{41}+(1+q)^2(1+q+q^2+q^3+q^4)e_5.
\end{equation}
\end{example}

\begin{figure}
$$\begin{tikzpicture}
\draw (0,0) node (1){$G=$};
\filldraw (0.634,0.5) circle (3pt) node[align=center,above] (1){1};
\filldraw (0.634,-0.5) circle (3pt) node[align=center,below] (2){2};
\filldraw (1.5,0) circle (3pt) node[align=center,above] (3){3};
\filldraw (2.366,0.5) circle (3pt) node[align=center,above] (4){4};
\filldraw (2.366,-0.5) circle (3pt) node[align=center,below] (5){5};
\draw (0.634,0.5)--(0.634,-0.5)--(1.5,0)--(0.634,0.5) (2.366,0.5)--(2.366,-0.5)--(1.5,0)--(2.366,0.5);
\draw (7.5,0) node (1) {$H=$};
\draw (8,0.866)--(8.5,0) (8,-0.866)--(8.5,0) (8.5,0)--(9.5,0);
\filldraw (8,0.866) circle (3pt) node[align=center,above] (1){1};
\filldraw (8,-0.866) circle (3pt) node[align=center,below] (2){2};
\filldraw (8.5,0) circle (3pt) node[align=center,above] (3){3};
\filldraw (9.5,0) circle (3pt) node[align=center,above] (4){4};
\draw (1,-2) node (){$X_G(\bm x)=4e_{32}+12e_{41}+20e_5$};
\draw (9,-2) node (){$X_H(\bm x)=e_{211}-2e_{22}+5e_{31}+4e_4$};
\end{tikzpicture}
$$
\caption{\label{fig:alphaomegaexample} The bowtie graph, the claw graph, and their chromatic symmetric functions}
\end{figure}

We now use the Alpha-Omega Lemma to prove several positivity results about the coefficients of particular partitions in all natural unit interval graphs.

\begin{corollary}\label{cor:twocols}
Let $\lambda$ be a partition with $\lambda_1\leq 2$. Then for every natural unit interval graph $G$, the coefficient $c_\lambda(q)$ of $e_\lambda$ in $X_G(\bm x;q)$ is a nonnegative polynomial.
\end{corollary}

\begin{proof}
By Lemma \ref{lem:alphaomega}, this coefficient $c_\lambda(q)$ is identically $0$ (which is clearly nonnegative) unless $\omega(G)\leq \lambda_1\leq 2$. However, by \eqref{eq:uicondition}, this means that the connected components of $G$ must be paths, for which Shareshian and Wachs proved that $X_G(\bm x;q)$ is $e$-positive \cite[Section 5]{chromposquasi}.
\end{proof}

In that case that $G$ is a disjoint union of paths, there is an explicit formula showing that the coefficient $c_{2^a1^b}(q)$ is $0$ unless the connected components of $G$ consist of exactly $b$ paths of odd size and some number $e$ of paths of even size, in which case $c_{2^a1^b}(q)=(1+q)^e$. In the case $q=1$ we can extend Corollary \ref{cor:twocols} to an even wider class of $\la$.

\begin{corollary}
Let $\lambda$ be a partition with $\lambda_1\leq 3$. Then for every natural unit interval graph $G$, the coefficient $c_\lambda$ of $e_\lambda$ in $X_G(\bm x)$ is nonnegative.
\end{corollary}

\begin{proof}
By the Alpha-Omega Lemma, this coefficient $c_\lambda$ is $0$ unless $\omega(G)\leq\lambda_1\leq 3$. However, Dahlberg \cite[Corollary 5.4]{trilad} proved that $X_G(\bm x)$ is $e$-positive for such graphs.
\end{proof}

In general, this argument shows that proving $e$-positivity of $X_G(\bm x)$ for graphs with clique number at most $k$ would prove nonnegativity of coefficients $c_\lambda$ with $\lambda_1\leq k$ for all natural unit interval graphs. In the next result we use $|\la|$ for the sum of the parts of the partition $\la$.

\begin{corollary}
\label{TopCor}
Let $\lambda$ be a partition of the form 
$\la=(\mu_1,\ldots,\mu_k,1,\ldots,1)$ for some partition $\mu=(\mu_1,\ldots,\mu_k)$.
 If $|\mu|\leq 6$, then for every natural unit interval graph $G$, the coefficient $c_\lambda(q)$ is nonnegative. If $|\mu|\leq 11$, then for every natural unit interval graph $G$, the coefficient $c_\lambda$ is nonnegative.
\end{corollary}

\begin{proof}
We will prove the first statement by induction on $n=|\lambda|$. We have checked by computer that the Shareshian--Wachs conjecture holds for all natural unit interval graphs $G$ with at most $10$ vertices, so we may assume that $n\geq 11$. We may also assume that $k\geq 2$ because Hwang \cite[Theorem 5.13]{chromtwoparthook} showed that $c_\lambda(q)$ is nonnegative for hooks. 

Suppose that $G$ is not connected, with $G=H_1\sqcup H_2$. The coefficient $c_\lambda(q)$ in $X_G(\bm x;q)$ arises from summing products of coefficients $c_\nu(q)$ in $X_{H_1}(\bm x;q)$ and $c_\rho(q)$ in $X_{H_2}(\bm x;q)$ over all ways of splitting the parts of $\lambda$ into partitions $\nu$ and $\rho$. By induction, such coefficients $c_\nu(q)$ and $c_\rho(q)$ are nonnegative polynomials, so the coefficient $c_\lambda(q)$ must be as well.

Now suppose that $G$ is connected. By \eqref{eq:uicondition}, $G$ must contain the edges $i(i+1)$ for $1\leq i\leq n-1$. But now the independence number of $G$ satisfies
$$
\alpha(G)\leq\lceil n/2\rceil<n-4\leq n-|\mu|+k=\lambda'_1,
$$
so by the Alpha-Omega Lemma, the coefficient $c_\lambda(q)$ is the zero polynomial in this case.

The argument is similar for the second statement. Guay-Paquet \cite[Page 8]{stanstemreduction} checked by computer that the Stanley--Stembridge conjecture holds for all natural unit interval graphs $G$ with at most $20$ vertices, so we may assume that $n\geq 21$. If $G$ is not connected, the coefficient $c_\lambda$ is nonnegative by induction on $n$, and if $G$ is connected, we have
$$
\alpha(G)\leq\lceil n/2\rceil<n-9\leq n-|\mu|+k=\lambda'_1,
$$
so by Lemma \ref{lem:alphaomega}, the coefficient $c_\lambda=0$. 
\end{proof}

The Alpha-Omega Lemma can also be used to translate results about particular partitions to results about particular natural unit interval graphs. As an example, we give another proof of $e$-positivity for graphs with independence number $2$.

\begin{proposition}
Let $G$ be a natural unit interval graph with $\alpha(G)=2$. Then $X_G(\bm x)$ is $e$-positive.
\end{proposition}

\begin{proof}
By Lemma \ref{lem:alphaomega}, the only nonzero coefficients $c_\lambda$ appearing in the $e$-expansion of $X_G(\bm x)$ are when $\la$ has at most two parts. If $\lambda=n$ has a single part, then Stanley showed that $c_n$ is nonnegative \cite[Theorem 3.3]{chromsym}.  If $\lambda$ has exactly two parts, then Abreu--Nigro \cite[Corollary 1.10]{chromhesssplitting} and Rok--Szenes \cite[Theorem 4.1]{chromeschers} independently proved that $c_\lambda$ is nonnegative. 
\end{proof}

In general, this argument shows that proving positivity of coefficients $c_\lambda$ with $\lambda'_1\leq k$ would prove $e$-positivity for all natural unit interval graphs with independence number at most $k$. 

Recall that an \emph{acyclic orientation} of $G=([n],E)$ is an assignment $O$ of directions to each edge of $G$ such that there are no directed cycles. 
Directed edges are also called {\em arcs} and an arc from $i$ to $j$ is denoted $i\to j$.
Vertex $i$ is a \emph{sink} of $O$ if there are no outgoing arcs from $i$.
The {\em ascent set} of $O$ is
$$
\Asc O = \{i\to j \mid i<j\}
$$
with {ascent number}
$$
\asc O =\#\Asc O.
$$

Stanley \cite[Theorem 3.3]{chromsym} proved that for an arbitrary graph $G$, the sum 
\begin{equation}
\label{eq:ao}
a_j:=\sum_{\lambda: \ \lambda'_1=j}c_\lambda
\end{equation}
is the number of acyclic orientations of $G$ with exactly $j$ sinks. Shareshian and Wachs \cite[Theorem 5.3]{chromposquasi} proved that if $G$ is a natural unit interval graph, we have
\begin{equation}
\label{eq:aoq}
a_j(q):=\sum_{\lambda: \ \lambda'_1=j}c_\lambda(q)=\sum_{O: \text{ exactly }j\text{ sinks}}q^{\asc O}.
\end{equation}

Therefore, even if $X_G(\bm x)$ is not $e$-positive, these sums of coefficients must be nonnegative. We can ask whether it is possible for this sum to be $0$ without every individual coefficient being $0$. We use the Alpha-Omega Lemma to show that this cannot be the case.

\begin{proposition}
Let $G=([n],E)$ be a graph and suppose that $c_\lambda\neq 0$ for some partition $\lambda$ with $\lambda'_1=j$. Then the sum $a_j$ in \eqref{eq:ao} is not $0$. Moreover, if $G$ is a natural unit interval graph and $c_\lambda(q)\neq 0$ for some partition $\lambda$ with $\lambda'_1=j$, then the sum $a_j(q)$ in \eqref{eq:aoq} is not the zero polynomial.
\end{proposition}

\begin{proof}
We will construct an acyclic orientation $O$ of $G$ with exactly $j$ sinks, which shows that $a_j$ is not $0$, and if $G$ is a natural unit interval graph, that $a_j(q)$ is not the zero polynomial.

First suppose that $G$ is connected. By Lemma \ref{lem:alphaomega} we have $\alpha(G)\geq j$, so there is an independent set $I$ in $G$ with $|I|=j$. 
We label the vertices of $G$ 
with $\bf 1$ through $\bf n$ as follows, where boldface numbers are being used to distinguish these labels from the ones given by the natural unit interval labeling. Label the vertices of $I$ with the numbers $\bf 1$ through $\bf j$ in any way, then successively select an unlabelled vertex adjacent to a labelled vertex and assign it the smallest unused label. Note that because $G$ is connected, every vertex receives a label. We now define $O$ by directing each edge from the larger label to the smaller. By construction, $O$ is an acyclic orientation.
The elements of $I$ are independent and have the smallest labels, 
so they are sinks of $O$. The vertices not in $I$ are adjacent to a previously-labelled vertex and so have an outgoing edge in $O$.  Thus they are not sinks.

Now suppose that $G$ has connected components $H_1,\ldots,H_k$. In order to have $c_\lambda(q)\neq 0$, there must be a way to split the parts of $\lambda$ into partitions $\nu(1),\ldots,\nu(k)$ where for every $1\leq r\leq k$, the coefficient $c_{\nu(r)}(q)$ of $e_{\nu(r)}$ in $X_{H_r}(\bm x;q)$ is nonzero. By the Alpha-Omega Lemma we have $\alpha(H_r)\geq\nu(r)'_1$, so there are independent sets $I_r$ in $H_r$ with $|I_r|=\nu(r)'_1$. But now the above argument shows that each $H_r$ has an acyclic orientation $O_r$ with exactly $\nu(r)'_1$ sinks, and therefore $G$ has an acyclic orientation $O$ with exactly $\nu(1)'_1+\cdots+\nu(k)'_1=\lambda'_1=j$ sinks.
\end{proof}

We now prove a stronger version of the Alpha-Omega Lemma
in the spirit of Greene's generalization~\cite{gre:est} of Schensted's theorem~\cite{sch:lid} about longest increasing and decreasing subsequences of a permutation. Let $\rho=\{B_1,\ldots,B_k\}$ be a collection of disjoint subsets of vertices in a graph $G=([n],E)$. We say that $\rho$ is an \emph{independent partition} if each $B_i$ is an independent set in $G$ and we say that $\rho$ is a \emph{clique partition} if each $B_i$ is a clique in $G$. We now define
\begin{align*}
\alpha_k(G)&=\text{the maximum number of vertices in an independent partition }\rho=\{B_1,\ldots,B_k\}\\
\omega_k(G)&=\text{the maximum number of vertices in a clique partition }\rho=\{B_1,\ldots,B_k\}.
\end{align*}

\begin{lemma}\label{lem:kalphaomega}
Let $G$ be a graph and let $X_G=\sum_\lambda c_\lambda e_\lambda$ be its chromatic symmetric function. If $c_\lambda\neq 0$, then for all $k$ we must have
\begin{enumerate}
\item $\alpha_k(G)\geq\lambda'_1+\cdots+\lambda'_k$, and 
\item $\omega_k(G)\leq\lambda_1+\cdots+\lambda_k$.
\end{enumerate}

Moreover, suppose that $G$ is a graph for which the chromatic quasisymmetric function is in fact symmetric, and let $X_G(\bm x;q)=\sum_\lambda c_\lambda(q)e_\lambda$. If $c_\lambda(q)\neq 0$, then again both inequalities hold.
\end{lemma}

\begin{proof}
As in the proof of Lemma \ref{lem:alphaomega}, we first consider the expansion $X_G(\bm x)=\sum_\mu a_\mu m_\mu$ in the basis of monomial symmetric functions. If $a_\mu\neq 0$, then there exists a proper colouring $\kappa$ of $G$ with $\mu_1$ $1$'s, $\mu_2$ $2$'s, and so on. Letting $B_i$ be the set of vertices coloured $i$, then $\rho=\{B_1,\ldots,B_k\}$ is precisely an independent partition of size $\mu_1+\cdots+\mu_k$, which means that 
$$
\alpha_k(G)\geq\mu_1+\cdots+\mu_k.
$$
On the other hand, let $\rho=\{C_1,\ldots,C_k\}$ be a clique partition of $G$ with $\omega_k(G)$ vertices. Because every colour is used at most once in each clique and therefore at most $k$ times in the vertices of $\rho$, we have
$$
\omega_k(G)=|C_1|+\cdots+|C_k|\leq \sum_{i-1}^n \min\{\mu_i,k\}=\mu'_1+\cdots+\mu'_k.
$$

Now if $c_\lambda\neq 0$, then there must be some $\mu$ with $\mu'\unlhd\lambda$ and $a_\mu\neq 0$, and so

$$
\alpha_k(G)\geq\mu_1+\cdots+\mu_k\geq\lambda'_1+\cdots+\lambda'_k
$$
and 
$$
\omega_k(G)\leq\mu'_1+\cdots+\mu'_k\leq\lambda_1+\cdots+\lambda_k.
$$

If $X_G(\bm x;q)$ is symmetric with $X_G(\bm x;q)=\sum_\mu a_\mu(q)m_\mu$ and $a_\mu(q)\neq 0$, then again there exists such a proper colouring, so again the inequalities hold.
\end{proof}

We can use Lemma \ref{lem:kalphaomega} to prove analogues of results obtained using the Alpha-Omega Lemma for other partitions $\la$. Recall that a \emph{cut vertex} of a connected graph $G$ is a vertex whose removal would disconnect $G$.

\begin{corollary}
Let $G=([n],E)$ be a connected natural unit interval graph with no cut vertex and let $\lambda$ be a partition of the form 
$\la=(\mu_1,\ldots,\mu_k,2,\ldots,2,1,\ldots,1)$ for some partition $\mu=(\mu_1,\ldots,\mu_k)$.
If $|\mu|\leq 4$, then the coefficient $c_\lambda(q)$ is nonnegative. If $|\mu|\leq 8$, then the coefficient $c_\lambda$ is nonnegative.
\end{corollary}

\begin{proof}
We first consider $c_\la(q)$ where we can assume, as in the proof of Corollary~\ref{TopCor}, that $n\ge11$.
By \eqref{eq:uicondition}, because $G$ is connected, $G$ must contain the edges $i(i+1)$ for $1\leq i\leq n-1$, and because $G$ has no cut vertex, $G$ must also contain the edges $(i-1)(i+1)$ for $2\leq i\leq n-1$. Now the largest possible independent partition of $G$ with two blocks is $\rho=\{\{1,4,7,\ldots\},\{2,5,8,\ldots\}\}$. So, using the bound on $n$,
$$
\alpha_2(G)\leq\lceil 2n/3 \rceil< n-2 \le n-|\mu|+2k = \lambda'_1+\lambda'_2,
$$
Thus, by Lemma \ref{lem:kalphaomega}, for such $\la$ we have  $c_\lambda(q)=0$.

Similarly, for $c_\la$ we need only check $n\ge 21$ as in Corollary~\ref{TopCor}'s demonstration.
The previously displayed inequalities hold with $n-2$ replaced by $n-6$.  So in this case, as before, $c_\la=0$ completing the proof.
\end{proof}

For each  natural unit interval graph $G=([n],E)$ we will now identify the unique dominance-minimal partition $\la_G$ for which 
$c_{\la_G}\neq 0$, which provides a sort of converse to the Alpha-Omega Lemma. 
Define a proper colouring 
$\ka_G$ of $G$ inductively as follows. Let $\ka_G(1)=1$ and if the values 
$\ka_G(1),\ldots,\ka_G(j-1)$ are defined, then let
\begin{equation}
\label{kappa}
\ka_G(j)=\min([n]\setminus\{\ka_G(i) \mid i<j \text{ and } ij\in E\}).
\end{equation}
In other words, we colour the vertices one at a time, always using the smallest available colour. Note that,
by~\eqref{eq:uicondition},
if $\ka_G(j)=k$ then the vertex $j$ has at least $(k-1)$ smaller neighbours and so, by \eqref{eq:uicondition}, belongs to a clique of size $k$. Therefore, the colouring $\kappa$ uses exactly $\omega(G)$ different colours. See Figure~\ref{fig:lexlargestcolouring} for an example.

Define the {\em type} of a coloring $\ka$ to be the weak composition 
$\al=(\al_1,\ldots,\al_m)$ where $m$ is the maximum value of a color and
$$
\al_i =\text{ the number of vertices colored $i$ by $\ka$}
$$
for $i\in[m]$.  We will use the notation $\mu_G$ for the type of $\ka_G$.  As the notation suggests, we will see in the next proposition that $\mu_G$ is indeed a partition.

\begin{figure}
\begin{equation*}
\begin{tikzpicture}
\draw (1,0) node (){$G=$};
\filldraw (2,0) circle (3pt) node[align=center,below] (1){1};
\filldraw (3,0) circle (3pt) node[align=center,below] (2){2};
\filldraw (4,0) circle (3pt) node[align=center,below] (3){3};
\filldraw (5,0) circle (3pt) node[align=center,below] (4){4};
\filldraw (6,0) circle (3pt) node[align=center,below] (5){5};
\draw (2,0)--(6,0);
\draw (2,0) -- (4,0) arc(0:180:1) --cycle;
\draw (4,0) -- (6,0) arc(0:180:1) --cycle;
\draw (1,-1) node () {$\kappa=$};
\draw (2,-1) node (){1};
\draw (3,-1) node (){2};
\draw (4,-1) node (){3};
\draw (5,-1) node (){1};
\draw (6,-1) node (){2};
\end{tikzpicture}
\end{equation*}
\caption{\label{fig:lexlargestcolouring} The bowtie graph $G$ and the proper colouring $\ka_G$}
\end{figure}
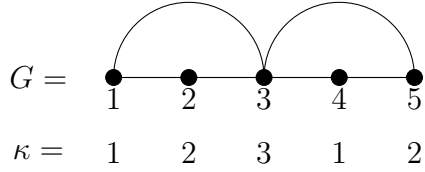

\begin{proposition}\label{prop:domsmallest}
Let $G=([n],E)$ be a natural unit interval graph. 
Then $\mu_G$ is a partition and $c_{\la_G},c_{\la_G}(q)\neq0$ where $\la_G=\mu_G'$.
Furthermore, if $c_\la\neq0$ then $\la\dom\la_G$.
Similarly, $c_\la(q)\neq0$ implies $\la\dom\la_G$.
\end{proposition}
\begin{proof}
To show that  $\mu_G$ is a partition it suffices to show, for all $k\ge 1$, that there is a matching $M$ from the vertices with color $k+1$ into the vertices with color $k$.  Suppose $\ka_G(j)=k+1$.  Then by the definition of $\ka_G$, vertex $j$ must be adjacent to vertices previously colored $1,\ldots,k$.  Since $G$ is a natural unit interval graph, the set of smaller neighbours of $j$ is a clique and so there is a unique vertex $i$ adjacent to $j$ with $i<j$ and $\ka_G(i)=k$.  Put $ij$ into  $M$ and do this for all the vertices colored $k+1$.  To show that this is a matching, we must prove that we can't have $\ka(j)=\ka(j')=k+1$ with $ij,ij'\in E$.  Without loss of generality $j<j'$.
So $i<j<j'$ which together with $ij'\in E$ implies $jj'\in E$.  But we can not have two vertices of the same color adjacent in $G$, the desired contradiction.

We now use induction on $n$ to prove that if $G$ has a proper colouring $\kappa$ of type $\alpha$, then
\begin{equation}\label{eq:domtype}(\mu_G)_1+\cdots+(\mu_G)_j\geq\alpha_1+\cdots+\alpha_j\end{equation}
for every $j$. This is clear if $n=1$, so suppose that $n\geq 2$. Let $\kappa_G(n)=k$. We first prove that \eqref{eq:domtype} holds for every $j\geq k$. Let $\tilde\mu$ and $\tilde\alpha$ be the types of the colourings $\kappa_G$ and $\kappa$ restricted to the vertices $\{1,\ldots,n-1\}$. Then by our induction hypothesis, we have for every $j\geq k$ that
$$(\mu_G)_1+\cdots+(\mu_G)_j=\tilde\mu_1+\cdots+\tilde\mu_j+1\geq \tilde\alpha_1+\cdots+\tilde\alpha_j+1\geq\alpha_1+\cdots+\alpha_j.$$
We now prove that \eqref{eq:domtype} holds for every $j<k$. Let $m$ be minimal such that the vertices $C=\{m,\ldots,n\}$ form a clique in $G$. Let $\tilde\mu$ and $\tilde\alpha$ be the types of the colourings $\kappa_G$ and $\kappa$ restricted to the vertices $\{1,\ldots,m-1\}$. By construction, $\kappa_G$ uses the colours $1,\ldots,k$ exactly once each in $C$, 
so $(\mu_G)_i=\tilde\mu_i+1$ for every $1\leq i\leq k$. Because $C$ is a clique, $\kappa$ uses every colour at most once each in $C$, so $\alpha_i\leq\tilde\alpha_i+1$ for every $i$. Therefore, by our induction hypothesis, we have for every $j<k$ that
$$(\mu_G)_1+\cdots+(\mu_G)_j=\tilde\mu_1+\cdots+\tilde\mu_j+j\geq \tilde\alpha_1+\cdots+\tilde\alpha_j+j\geq\alpha_1+\cdots+\alpha_j.$$
Therefore, \eqref{eq:domtype} holds for every $j$. In particular, if the type of $\kappa$ is a partition $\nu$, then 
$\mu_G\unrhd\nu$.

Now the change of basis from monomial to elementary symmetric functions has the form 
$$m_\nu=\sum_\rho A_{\nu,\rho}e_\rho,$$
where $A_{\nu,\rho}=0$ unless $\nu'\unlhd\rho$, and $A_{\nu,\nu}=1$. Therefore, $c_{\la_G}\neq 0$ and $c_{\la_G}(q)\neq 0$.  And if $\rho$ is a partition for which $c_\rho\neq 0$ or $c_\rho(q)\neq 0$, then there must be a partition $\nu$ for which $a_\nu\neq 0$ and $\nu'\unlhd\rho$, so 
$$\rho\unrhd\nu'\unrhd\mu_G'=\lambda_G
$$
which completes the proof.
\end{proof}

\begin{example}
The colouring $\kappa_G$  for the bowtie graph is shown in Figure \ref{fig:lexlargestcolouring}. It has type $\la_G=(2,2,1)$.  So Proposition \ref{prop:domsmallest} implies that the unique dominance-minimal partition for which $c_\lambda\neq 0$ is $\lambda=\mu'=32$. This agrees with the expression for $X_G(\bm x)$ given in Figure \ref{fig:alphaomegaexample}.
\end{example}

We conclude this section by using the monomial change of basis  to prove a divisibility result.
The {\em closed neighborhood} of a vertex $v$ of $G$ is the set consisting of $v$ and all its adjacent vertices.
Let us say that vertices $u$ and $v$ of a graph $G=([n],E)$ are \emph{equivalent} if they have the same closed neighbourhood. Note that equivalent vertices must be adjacent. We now prove a divisibility result about the coefficients $c_\lambda$ and $c_\lambda(q)$. For an integer $k$ we define the standard $q$-analogues
$$
[k]_q=1+q+q^2+\cdots+q^{k-1}=\frac{q^k-1}{q-1}
$$
and 
$$
[k]_q!=[k]_q[k-1]_q\cdots[2]_q[1]_q.
$$

\begin{proposition}
Let $G=([n],E)$ and let $C_1,\ldots,C_k$ be the equivalence classes of the vertex set $[n]$. Then every coefficient $c_\lambda$ is divisible by $|C_1|!\cdots|C_k|!$. If $G$ is a graph for which $X_G(\bm x;q)$ is symmetric, then every coefficient $c_\lambda(q)$ is divisible by $[|C_1|]_q!\cdots[|C_k|]_q!$.
\end{proposition}

\begin{proof}
We will show that every coefficient $a_\mu$ in the expansion $X_G(\bm x)=\sum_\mu a_\mu m_\mu$ is divisible by $|C_1|!\cdots|C_k|!$, from which the first statement follows because the coefficients in the change of basis from monomial to elementary symmetric functions are all integers. Recall that $a_\mu$ is the number of proper colourings of $G$ of type $\mu$, meaning they use $\mu_1$ $1$'s, $\mu_2$ $2$'s, and so on.

We first note that because each $C_i$ is a clique in $G$, a proper colouring $\kappa$ of $G$ must colour the vertices of $C_i$ with $|C_i|$ distinct colours. Let us say that proper colourings $\kappa$ and $\kappa'$ are \emph{similar} if they use the same colours on every $C_i$, in other words we have
$$
\{\kappa(v): \ v\in C_i\}=\{\kappa'(v): \ v\in C_i\}\text{ for every }1\leq i\leq k.
$$
Let us call a proper colouring $\kappa$ \emph{initial} if it has no ascents on any $C_i$, in other words we have
\begin{equation}\label{eq:initial}
\kappa(u)>\kappa(v)\text{ whenever }u<v\text{ and }u,v\in C_i\text{ for some }1\leq i\leq k.
\end{equation}
Now every proper colouring of $G$ is similar to a unique initial colouring given by permuting the colours in each $C_i$ so that~\eqref{eq:initial} holds. Conversely, given an initial proper colouring $\kappa$, there are exactly $|C_1|!\cdots|C_k|!$ proper colourings $\kappa'$ similar to it, given by permuting the colours in each $C_i$ in all possible ways. Therefore, the coefficient $a_\mu$ is equal to $|C_1|!\cdots|C_k|!$ multiplied by the number of initial proper colourings of type $\mu$.

For the second statement, we note that when we obtain a proper colouring $\kappa'$ from an initial proper colouring $\kappa$ by applying a permutation $\sigma_i$ to the vertices in $C_i$, the number of ascents increases by exactly \begin{equation}\text{des}(\sigma_i)=|\{(u,v): \ u,v\in C_i, \ u<v, \ \sigma_i(u)>\sigma_i(v)\}|.\end{equation}
We also note that $\sum_{\sigma\in S(C_i)}q^{\text{des}(\sigma)}=[|C_i|]_q!$. Therefore, we have
$$
a_\mu(q)=\sum_{\kappa\text{ of type }\mu}q^{\asc\kappa}=[|C_1|]_q!\cdots[|C_k|]_q!\sum_{\kappa\text{ initial of type }\mu} q^{\asc\kappa}.
$$
In particular, the coefficient $a_\mu(q)$ is divisible by $[|C_1|]_q!\cdots[|C_k|]_q!$ and so the coefficient $c_\mu(q)$ must be as well.
\end{proof}

\begin{figure}
\begin{tikzpicture}
\draw (1.25,0) node (){$K_{6,6}=$};
\filldraw (2,0) circle (3pt) node[align=center,above] (5){1};
\filldraw (2.5,0.866) circle (3pt) node[align=center,above] (6){2};
\filldraw (3.5,0.866) circle (3pt) node[align=center,above] (7){4};
\filldraw (2.5,-0.866) circle (3pt) node[align=center,below] (8){3};
\filldraw (3.5,-0.866) circle (3pt) node[align=center,below] (9){5};
\filldraw (4,0) circle (3pt) node[align=center,above] (10){6};

\draw (2,0)--(6,0);
\draw (2,0) -- (2.5,0.866) -- (3.5,0.866) -- (4,0) -- (3.5,-0.866) -- (2.5,-0.866) -- (2,0);
\draw (2,0) -- (3.5,-0.866) -- (3.5,0.866) -- (2,0);
\draw (4,0) -- (2.5,0.866) -- (2.5,-0.866) -- (4,0);
\draw (2.5,0.866) -- (3.5,-0.866);
\draw (2.5,-0.866) -- (3.5,0.866);
\filldraw (4.5,0.866) circle (3pt) node[align=center,above] (6){7};
\filldraw (5.5,0.866) circle (3pt) node[align=center,above] (7){9};
\filldraw (4.5,-0.866) circle (3pt) node[align=center,below] (8){8};
\filldraw (5.5,-0.866) circle (3pt) node[align=center,below] (9){10};
\filldraw (6,0) circle (3pt) node[align=center,above] (10){11};

\draw (4,0) -- (4.5,0.866) -- (5.5,0.866) -- (6,0) -- (5.5,-0.866) -- (4.5,-0.866) -- (4,0);
\draw (4,0) -- (5.5,-0.866) -- (5.5,0.866) -- (4,0);
\draw (6,0) -- (4.5,0.866) -- (4.5,-0.866) -- (6,0);
\draw (4.5,0.866) -- (5.5,-0.866);
\draw (4.5,-0.866) -- (5.5,0.866);
\draw(4,-2) node{$X_{K_{6,6}}(\bm x;q)=[5]_q![5]_q!(q^5e_{65}+q^4[3]_qe_{74}+q^3[5]_qe_{83}+q^2[7]_qe_{92}+q[9]_qe_{(10)1}+[11]_qe_{(11)})$};
\end{tikzpicture}
\caption{\label{fig:kchain66example} The graph $K_{6,6}$ and its chromatic quasisymmetric function $X_{K_{6,6}}(\bm x;q)$}
\end{figure}

\begin{example}
\label{ex:kab}
Let $K_{a,b}$ denote the graph with $n=a+b-1$ vertices obtained by gluing a clique of size $a$ and a clique of size $b$ at a single vertex, to be precise,
$$
K_{a,b}=([n],\{ij: \ 1\leq i\leq j\leq a\}\cup\{ij:a\leq i\leq j\leq n\}).
$$
Then the equivalence classes of the vertices of $K_{a,b}$ are
$C_1=\{1,\ldots,a-1\}$, $C_2=\{a\}$, and $C_3=\{a+1,\ldots,n\}$. The chromatic quasisymmetric function $X_{K_{a,b}}(\bm x;q)$ has the explicit formula \cite[Corollary 4.14]{qforesttriples}
\begin{equation}
\label{eq:kab}
  X_{K_{a,b}}(\bm x;q)=[a-1]_q![b-1]_q!\sum_{k=\max\{a,b\}}^nq^{n-k}[2k-n]_qe_{k(n-k)}.  
\end{equation}
Indeed, every coefficient $c_\lambda(q)$ is divisible by $[|C_1|]_q![|C_2|]_q![|C_3|]_q!=[a-1]_q![1]_q![b-1]_q!$.
\end{example}

\section{The Schur basis}\label{sec:Schur}

If $G$ is a natural unit interval graph then equation~\eqref{eq:aoq} gives an interpretation of the coefficient of $e_n$ in $X_G(\bm x;q)$ as a generating function for the acyclic orientations of $G$ with one sink.
In this section, we use the Schur basis to give another interpretation for the coefficient of $e_n$ in terms of tableaux. 
Throughout this section $G=([n],E)$ will be a natural unit interval graph. Gasharov \cite[Theorem 4]{chrompos31} proved a combinatorial formula for the Schur expansion of $X_G(\bm x)$ in terms of certain tableaux and Shareshian and Wachs \cite[Theorem 6.3]{chromposquasi} proved a $q$-analogue of this result. We will write our Young diagrams in English notation.

\begin{definition}
A \emph{$G$-tableau of shape $\lambda$} is a bijective filling $T$ of the diagram of $\lambda$ with the numbers $1$ through $n$ such that the following two conditions hold.
\begin{enumerate}
\item[(C1)] If an entry $i$ is directly to the left of an entry $j$ in the same row, then $i<j$ and $ij\notin E$.
\item[(C2)] If an entry $i$ is directly above an entry $j$ in the same column, then either $i<j$ or $ij\in E$.
\end{enumerate}
Let $\text{GTab}_\lambda$ denote the set of $G$-tableaux of shape $\lambda$. An \emph{inversion} of $T$ is a pair of entries $(i,j)$ such that
\begin{enumerate}
    \item[(I1)] $i<j$,
    \item[(I2)] $ij\in E$, and
    \item[(I3)] $i$ is in a lower row than $j$.    
\end{enumerate}
We denote by $\Inv T$ the set of inversions of $T$ and $\inv T=\#\Inv T$.
\end{definition}

\begin{theorem}[\cite{chromposquasi}]
\label{thm:Gtab}
The chromatic quasisymmetric function of a natural unit interval graph $G$ satisfies
\begin{equation}
\label{eq:s}
X_G(\bm x;q)=\sum_\lambda\sum_{T\in\text{GTab}_\lambda}q^{\inv T}s_\lambda.
\end{equation}
In particular, $X_G(\bm x;q)$ is Schur-positive.
\end{theorem}

\begin{figure}
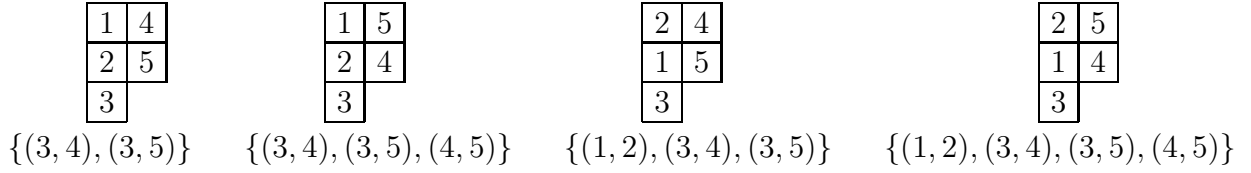

\begin{align*}
&\hspace{30pt} \tableau{1&4\\2&5\\3} \hspace{60pt} \tableau{1&5\\2&4\\3} \hspace{90pt} \tableau{2&4\\1&5\\3} \hspace{120pt} \tableau{2&5\\1&4\\3}\\
&\{(3,4),(3,5)\} \hspace{20pt} \{(3,4),(3,5),(4,5)\} \hspace{20pt} \{(1,2),(3,4),(3,5)\} \hspace{20pt} \{(1,2),(3,4),(3,5),(4,5)\}
\end{align*}
\caption{\label{fig:Gtab} The four $G$-tableaux of shape $\lambda=221$ for the bowtie graph $G$}
\end{figure}

\begin{example}
Let $G$ be the bowtie graph from Figure \ref{fig:alphaomegaexample}. The four $G$-tableaux of shape $\lambda=221$ and their inversion sets are given in Figure \ref{fig:Gtab}. By Theorem \ref{thm:Gtab}, the coefficient of $s_{221}$ in $X_G(\bm x;q)$ is $(q^2+2q^3+q^4)$. 
\end{example}

Now we can calculate coefficients in the elementary basis by first using Theorem \ref{thm:Gtab} and then converting Schur functions to elementary symmetric functions using the dual Jacobi--Trudi determinantal identity \cite[Theorem 4.5.1]{symgroup}, \cite[Corollary 7.16.2]{enum2},
\begin{equation}\label{eq:jt}
s_\lambda=\det(e_{\lambda'_i-i+j})_{i,j=1}^{\la_1}.
\end{equation}
In particular, the coefficient of $e_n$ in $s_\lambda$ is nonzero only if $\lambda$ is a \emph{hook}, meaning a partition of the form $\lambda=k1^{n-k}$ for some $1\leq k\leq n$, in which case this coefficient is $(-1)^{k-1}$. Let 
$$
\HT(G)=\{ T \mid \text{$T$ is a $G$-tableau of hook shape}\},
$$
and if $T\in\HT(G)$ has shape $k1^{n-k}$
then define its {\em sign} to be
$$
\sgn T=(-1)^{k-1}. 
$$
By Theorem \ref{thm:Gtab} and \eqref{eq:jt}, the coefficient of $e_n$ in $X_G(\bm x;q)$ is
\begin{equation}\label{eq:cnq}
c_n(q)=\sum_{T\in\HT(G)}(\sgn T) q^{\inv T}.
\end{equation}

We now define a sign-reversing involution $\varphi_G$ on $\HT(G)$ to obtain a positive combinatorial description for $c_n(q)$. The idea will be to move entries between the \emph{arm} and \emph{leg} of $T$, which are the first row and first column of $T$, respectively, excluding the corner entry in the $(1,1)$-cell. We will only move entries if we can preserve the inversion set.

\begin{definition}
Let $T\in\HT(G)$ be a $G$-tableau. An entry $j$ in the arm of $T$ is \emph{movable} if it can be moved to the leg of $T$ to produce a tableau 
$T^j$ such that
\begin{enumerate}
    \item[(M1)] $T^j\in\HT(G)$, and
    \item[(M2)] $\Inv T^j=\Inv T$.
\end{enumerate}
Similarly,  entry $j$ is {\em movable} from the leg to the arm if the resulting tableau $T^j$ satisfies (M1) and (M2).

\end{definition}

We will see shortly that if $j$ is movable then the position to which $j$ can be moved, and hence $T^j$, are uniquely determined.

\begin{example}

In the first line of Figure~\ref{move} we see a natural unit interval graph, $G$, and one of its $G$-tableaux, $T$, of hook shape.  It is easy to check that $\Inv T =\{ (2,3), (4,5)\}$.  Both $3$ and $5$ are movable and the resulting $G$-tableaux are displayed in the second line.  The elements $2$ and $4$ are not 
movable.

\end{example}

\begin{figure}
\begin{center}
\begin{tikzpicture}
\draw(-2,0) node{$G=$};
\filldraw(-1,0) circle(.1);
\draw(-1,-.3) node{$1$};
\filldraw(0,0) circle(.1);
\draw(0,-.3) node{$2$};
\filldraw(1,0) circle(.1);
\draw(1,-.3) node{$3$};
\filldraw(2,0) circle(.1);
\draw(2,-.3) node{$4$};
\filldraw(3,0) circle(.1);
\draw(3,-.3) node{$5$};
\draw plot [smooth, tension=2] coordinates {(0,0) (1,.7) (2,0)};
\draw (-1,0)--(3,0);
\draw(5.5,0) node{$T=$};
\draw (7,0) node{\tableau{1 & 3 & 5\\ 2\\ 4}};
\draw(-.5,-2) node{$T^3=$};
\draw(1,-2) node{\tableau{1&5\\ 3\\ 2 \\ 4}};
\draw(3.5,-2) node{$T^5=$};
\draw(5,-2) node{\tableau{1&3\\ 2\\ 5 \\ 4}};
\end{tikzpicture}
\end{center}
\caption{\label{move} A natural unit interval graph $G$, a $G$-tableau $T$, and two $T^j$}
\end{figure}
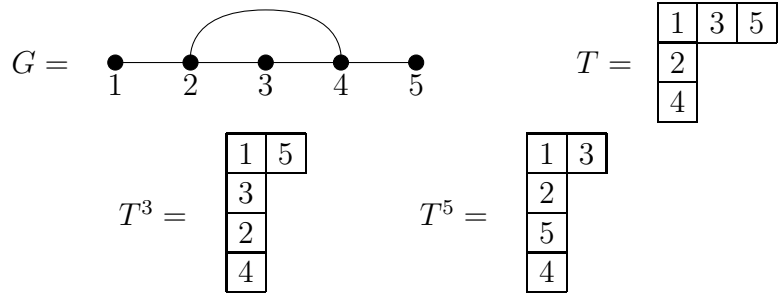

\begin{lemma} \label{lem:legtoarm}
Let $T\in\HT(G)$ be a $G$-tableau. 

\begin{enumerate}
\item[(a)] An entry $j$ in the leg of $T$ is movable if and only if every entry $i$ of $T$ in a row above $j$ satisfies $ij\notin E$. 
\item[(b)] If the entry $j$ is movable, there is a unique position to which it can be moved and every entry $i$ above $j$ in the leg of $T$ satisfies $i<j$. 
\end{enumerate}

\end{lemma}

\begin{proof}
(a) If $j$ is movable, then we must have $ij\notin E$ for every entry $i$ in the first row of $T$ by (C1) and because $G$ is a natural unit interval graph. We also must have $ij\notin E$ for every entry $i$ in the leg of $T$ above $j$ because otherwise there would be a change in inversions. 

Conversely, suppose that $ij\notin E$ for every entry $i$ above $j$ and let $T^j$ be the tableau obtained by moving $j$ to the arm so that it is strictly increasing, so (C1) holds. We have $\Inv T^j=\Inv T$ because the only change in relative positions occurred between non-adjacent vertices of $G$. If there are entries $x$ and $y$ respectively directly above and below $j$ in $T$, we cannot have $x>y$ and $xy\notin E$ because then either $j>x>y$ and $yj\notin E$ because $G$ is a natural unit interval graph, violating (C2), or $j<x$  and $xj\notin E$ by assumption, again violating (C2). Therefore, the tableau $T^j$ is indeed a $G$-tableau and the entry $j$ is movable.

(b) First note that uniqueness holds because the first row of $T^j$ must be strictly increasing by (C1).  Next, suppose that $j$ is movable and there is an entry $i>j$ above $j$ in the leg of $T$. We have shown that $ij\notin E$ so $i$ must not be directly above $j$ by (C2). Therefore, there is some other entry $i'$ directly above $j$, and $i'j\not\in E$ because $j$ is movable, so $i'<j$. But now because $i>j>i'$, there must be entries $x$ and $y$ in the leg of $T$ above $j$ with $x$ directly above $y$ and $x>j>y$. But because $j$ is movable we have $jx\notin E$, and $xy\notin E$ because $G$ is a natural unit interval graph, which violates (C2).
\end{proof}

\begin{lemma} \label{lem:armtoleg}
Let $T\in\HT(G)$ be a $G$-tableau. Every entry in the arm of $T$ is movable and can be moved to a unique position.
\end{lemma}

\begin{proof}
Let $j$ be an entry in the arm of $T$ and let $x_1,\ldots,x_\ell$ be the entries of the first column of $T$ from top to bottom. By (C1) and the fact that $G$ is a natural unit interval graph, we have $x_1<j$ and $x_1j\notin E$. Therefore, there is some maximal index $i$ with $1\leq i\leq \ell$ such that we have $x_t<j$ and $x_tj\notin E$ for every $1\leq t\leq i$. We claim that the entry $j$ is movable to the position directly below the entry $x_i$. The first row of the resulting tableau $T^j$ satisfies (C1) because $G$ is a natural unit interval graph. We have $x_i<j$ by definition and we cannot have both $x_{i+1}<j$ and $x_{i+1}j\notin E$ by maximality of $i$, so the first column of $T^j$ satisfies (C2). We have $\Inv T^j=\Inv T$ because the entries in the arm of $T^j$ and the entries $x_1,\ldots,x_i$ that are now above $j$ are not adjacent to $j$ so there is no change in inversions. 

Finally, the entry $j$ could not have moved to a higher position because (C2) would not hold by definition of $i$. Also, the entry $j$ could not have moved to a lower position because then it would be movable in $T^j$, which means by Lemma \ref{lem:legtoarm} that $x_{i+1}<j$ and $x_{i+1}j\notin E$, contradicting maximality of $i$.
\end{proof}

Note that by Lemma \ref{lem:legtoarm} and Lemma \ref{lem:armtoleg}, movable entries can be moved to a unique position so the tableau $T^j$ is well-defined, as promised.

\begin{definition}
Let $T\in\HT(G)$ be a $G$-tableau. We define the $G$-tableau
\begin{equation}
\varphi_G(T)=\begin{cases} T&\text{ if }T\text{ has no movable entry,}\\ T^j&\text{ if }T\text{ has smallest movable entry }j.\end{cases}
\end{equation}
We denote by $\Fix\varphi_G$ the set of \emph{fixed points} of $\varphi_G$, meaning those tableaux $T$ for which $\varphi_G(T)=T$.
\end{definition}

We now show that the map $\varphi_G:\HT(G)\to\HT(G)$ is an inv-preserving, sign-reversing involution.

\begin{lemma}\label{lem:signrev}
We have $\varphi_G(\varphi_G(T))=T$ and $\inv \varphi_G(T)=\inv T$ for every $T\in\HT(G)$. Also
$$
\sgn \varphi_G(T) =
\begin{cases}
1   & \text{if $T\in\Fix\varphi_G$,}\\
-\sgn T & \text{else.}
\end{cases}
$$
\end{lemma}

\begin{proof}
This first statement is clear if $T$ has no movable entry because $\varphi_G(T)=T$, so suppose that $T$ has smallest movable entry $j$, so that $\varphi_G(T)=T^j$. We have $\inv T^j=\inv T$ by definition of $\varphi_G$.
By construction, the entry $j$ of $T^j$ is movable and the unique position to which it can be moved was its original position in $T$. However, we need to check that the entry $j$ of $T^j$ is in fact the smallest movable entry. We will show that if a smaller entry $i<j$ of $T^j$ is movable, then $i$ was movable in $T$, contradicting minimality of $j$. 

By Lemma \ref{lem:armtoleg} (a), every entry in the arm of a hook tableau is movable. So if $i$ were in the arm of $T^j$ then it would have been movable in $T$.   Now  suppose that $i$ is in the leg of $T^j$. By Lemma \ref{lem:legtoarm}, the condition for $i$ to be movable is that it is not adjacent to any entries in a higher row. 
So  $i$  movable in $T^j$ but not $T$ means that it is not adjacent to everything above it in $T^j$ but adjacent to something above in  $T$.  This can only happen if $j$ is in the arm of $T$, moves to the leg of $T^j$ below $i$, and $ij\in E$.  But $j$ being movable in $T^j$ implies $ij\notin E$ which is a contradiction.

Finally, if $T\in\Fix\varphi_G$, then $T$ has no movable entry.
By Lemma \ref{lem:armtoleg}, there are no elements in the arm so that $T$ has shape $1^n$ and $\sgn T =1$. If $T\notin\Fix\varphi_G$, then $\varphi_G$ changes the number of columns in passing from $T$ to $\varphi_G(T)$ by exactly one.  It follows that $\sgn \varphi_G(T) = -\sgn T$ which completes the proof.
\end{proof}

\begin{corollary}
Let $G$ be a natural unit interval graph. Then the coefficient of $e_n$ in $X_G(\bm x;q)$ is given by
\begin{equation}
c_n(q)=\sum_{T\in\Fix\varphi_G}q^{\inv T}.
\end{equation}
\begin{proof}
By Lemma \ref{lem:signrev}, the map $\varphi_G$ is an inversion-preserving bijection between the positively-signed non-fixed points and the negatively-signed non-fixed points, so the sum in \eqref{eq:cnq} is equal to the sum over the fixed points, which are positively signed.
\end{proof}
\end{corollary}

We now describe the fixed points of $\varphi_G$ to get an explicit formula for $c_n(q)$ which will follow from Corollary~\ref{qcor} below.

\begin{theorem}\label{thm:coefen}
Let $G=([n],E)$ be a natural unit interval graph. For $2\leq k\leq n$, let $b_k$ denote the number of smaller neighbours of vertex $k$ in $G$. Then the coefficient of $e_n$ in the chromatic quasisymmetric function $X_G(\bm x;q)$ is 
\begin{equation}
c_n(q)=[n]_q[b_2]_q[b_3]_q\cdots[b_n]_q.
\end{equation}
\end{theorem}

By considering the reverse graph obtained by relabeling vertex $k$ to $n-k+1$, we could equivalently think about the number of larger neighbours of each vertex. 

\begin{example}
For the bowtie graph $G$ from Figure \ref{fig:alphaomegaexample}, the coefficient of $e_5$ in $X_G(\bm x;q)$ is \begin{equation}
c_5(q)=[5]_q[2]_q[2]_q=(1+q+q^2+q^3+q^4)(1+q)^2.
\end{equation}
\end{example}

We will use induction on $n$ to enumerate the fixed points of $\varphi_G$. Suppose that $n\geq 2$ and let $G'$ denote the graph $G$ with the vertex $n$ removed. For a tableau $T$ of column shape, we denote by $T_i$ the $i$-th entry from the top.

\begin{lemma}\label{lem:fixind}
If $b_n=0$, then $\Fix\varphi_G=\emptyset$. Otherwise, the tableaux $T\in\Fix\varphi_G$ arise precisely by taking a tableau $T'\in\Fix\varphi_{G'}$ and inserting the entry $n$ either at the bottom, directly above a neighbour of $n$ other than the one occurring highest in $T'$, or at the top if $T'_1$ is a neighbour of $n$. 
\end{lemma}

\begin{proof}
If $b_n=0$, then a $G$-tableau $T\in\text{GTab}_{1^n}$ must have the entry $n$ at the bottom to satisfy (C2).
But then the $n$ is movable by Lemma \ref{lem:legtoarm}. So assume that $b_n\geq 1$. A tableau $T$ constructed from $T'$ in one of the three given ways will satisfy (C2) because $n$ is either at the bottom or directly above a neighbour, and will have no movable elements by Lemma \ref{lem:legtoarm} because $n$ is either at the top or below a neighbour, so we have $T\in\Fix\varphi_G$. We now show that a tableau $T\in\Fix\varphi_G$ must arise in this way. 

Let $T'$ denote the tableau obtained from $T$ by removing the entry $n$. 
We must first show that $T'$ still satisfies (C2).  Suppose, to the contrary, that 
there are entries $x$ and $y$ respectively above and below the entry $n$ in $T$ with $x>y$ and $xy\notin E$. But since  $G$ is a natural unit interval graph this would force $yn\notin E$ since $n>y$.  So $T$ would not satisfy (C2), a contradiction.

Also we claim that  $T'$ has no movable elements. Suppose that $T'$ has a movable element  $j$. By Lemma \ref{lem:legtoarm}, every element $i$ above $j$ in $T'$ must satisfy $i<j$ and $ij\notin E$. Because $j$ is not movable in $T$, the entry $n$ must be above $j$ in $T$ and $jn\in E$ by Lemma \ref{lem:legtoarm} (a). Now every entry $i$ above $n$ in $T$ satisfies $i<j<n$ which implies $in\notin E$ because $G$ is a natural unit interval graph. 
So, by Lemma \ref{lem:legtoarm} (a) again, the only way for $n$ to not be movable in $T$ is if 
$T_1=n$. If $T_2=j$, then $T'_1=j$ so $j$ is not movable in $T'$.  Otherwise, $T_2$ must be some other entry $k$. However, because $k$ is above $j$ in $T'$, we have shown that $k<j$ and $kj\notin E$. 
This  forces $k<n$ and $kn\notin E$ so that $T$ does not satisfy (C2). 
 This final contradiction shows that no element of $T'$ is movable.

Therefore $T'\in\Fix\varphi_G$ and indeed $T$ arises by inserting the entry $n$ somewhere in $T'$. In order to satisfy (C2), the entry $n$ must be inserted either at the bottom or directly above a neighbour of $n$, and in order for $n$ to not be movable, by Lemma \ref{lem:legtoarm}, it must be either at the top or have a neighbour above it.
\end{proof}

Theorem \ref{thm:coefen} follows from the next corollary by summing over $1\leq j\leq n$.

\begin{corollary}
\label{qcor}
Let $G=([n],E)$ be a natural unit interval graph. Then for every $1\leq j\leq n$, we have
\begin{equation}
\sum_{T\in\Fix\varphi_G, \ T_1=j}q^{\inv T}=q^{j-1}[b_2]_q\cdots[b_n]_q.
\end{equation}
\end{corollary}

\begin{proof}
We use induction on $n$, the case of $n=1$ trivial, so suppose that $n\geq 2$. If $b_n=0$, then the statement holds by Lemma \ref{lem:fixind}, so assume that $b_n\geq 1$ and note that because $G$ is a natural unit interval graph, the neighbours of $n$ are exactly $n-b_n,\ldots,n-1$. If $j<n$, then by Lemma \ref{lem:fixind}, every tableau $T\in\Fix\varphi_G$ with $T_1=j$ arises from a tableau $T'\in\Fix\varphi_{G'}$ by inserting the entry $n$ in one of $b_n$ possible positions. Inserting $n$ in the $i$-th place from the bottom results in $(i-1)$ new inversions, so by our induction hypothesis, we have
$$\sum_{T\in\Fix\varphi_G, \ T_1=j}q^{\inv T}=(1+q+\cdots+q^{b_n-1})\sum_{T'\in\Fix\varphi_{G'}, \ T'_1=j}q^{\inv T}=q^{j-1}[b_2]_q\cdots[b_n]_q.$$
If $j=n$, then by Lemma \ref{lem:fixind}, every tableau $T\in\Fix\varphi_G$ with $T_1=n$ arises from a tableau $T'\in\Fix\varphi_G$ with $n-b_n\leq T'_1\leq n-1$ by inserting the entry $n$ at the top. Doing so produces $b_n$ new inversions, so by our induction hypothesis, we have
\begin{align*}\sum_{T\in\Fix\varphi_G, \ T_1=n}q^{\inv T}&=\sum_{i=n-b_n}^{n-1}\sum_{T'\in\Fix\varphi_{G'}, \ T'_1=i}q^{\inv T+b_n}\\&=q^{b_n}\sum_{i=n-b_n}^{n-1}q^{i-1}[b_2]_q\cdots[b_{n-1}]_q=q^{n-1}[b_2]_q\cdots[b_n]_q,\end{align*}
which completes the proof.
\end{proof}

When $q=1$, a similar formula holds for a wider family of graphs. 
If $G=(V,E)$ and $v\in V$ then $N(v)$ will denote the (open) neighborhood of $v$ which is all vertices adjacent to $v$.
We say that a graph $G$ has a {\em perfect elimination order(peo)} if there is a permutation $v_1, v_2,\ldots, v_n$ of $V$ such that for all $k\in[n]$ the vertices of $N(v_k)\cap\{v_1,v_2,\ldots,v_{k-1}\}$ form a clique. This condition is equivalent to $G$ being {\em chordal} which means that every cycle $C$ of length at least $4$ has a chord, that is, an edge of $G$ connecting two vertices of $C$ which are not adjacent along the cycle. It is easy to see that if $G$ is a natural unit interval graph then $1,2,\ldots,n$ is a perfect elimination order. But there are graphs with a peo, for example, $k\geq 3$ triangles with a vertex of each identified, such that no labeling of the vertices with $[n]$ gives a natural unit interval graph.

\begin{theorem}
\label{e_n:q=1}
Let $G$ have a perfect elimination order $v_1,v_2,\ldots, v_n$.  
 For $2\leq k\leq n$, let 
$$
b_k=\#\{ v_j \mid \text{$v_j\in N(v_k)$ and $j<k$}\}.
$$
Then the coefficient of $e_n$ in $X_G(\bm x)$ is the product
$$
c_n=n b_2 b_3 \cdots b_n.
$$
\end{theorem}
\begin{proof}
By Stanley's Theorem on acyclic orientations \cite[Theorem 3.3]{chromsym}, $c_n$ is equal to the number of acyclic orientations of $G$ with a unique sink. Greene and Zaslavsky \cite[Theorem 7.3]{interpwhitney} showed that the number of acyclic orientations of $G$ with unique sink $v$ is independent of $v$ and equal to the absolute value of the linear coefficient of the chromatic polynomial $\chi(G;t)$. We make use of the perfect elimination order to count the number of ways to colour the vertices of $G$ in the order $v_1,v_2,\ldots,v_n$ using $t$ colours. There are $t$ choices for the colour of $v_1$. For every $2\leq k\leq n$, the $b_k$ neighbours of $v_k$ coloured already form a clique and so were given different colours.
It follows that there are $(t-b_k)$ choices for the colour of $v_k$ and
$$\chi(G;t)=t(t-b_2)\cdots(t-b_n).$$
Thus the number of acyclic orientations with unique sink $v$ is $b_2b_3\cdots b_n$. Since there are $n$ choices of unique sink $v$, there are $nb_2b_3\cdots b_n$ acyclic orientations with a unique sink and the theorem follows from Stanley's result.
\end{proof}

We conclude this section by relating $G$-tableaux of column shape to the set $\mathcal O(G)$ of acyclic orientations of $G$. Given a $G$-tableau $T\in\text{GTab}_{1^n}$, let $\psi(T)$ be the acyclic orientation of $G$ where every edge $ij$ is oriented from the entry occurring lower in $T$ to the entry occurring higher.

\begin{theorem}\label{thm:aotabs}

The map $\psi:\text{GTab}_{1^n}\rightarrow \mathcal O(G)$ is a bijection. If $\psi(T)=O$, then 
\begin{enumerate}
    \item[(a)] $\Inv T=\Asc O$,
    \item[(b)] the sinks of $O$ are precisely the entry $T_1$ along with the movable elements of $T$, and
    \item[(c)] the smallest sink of $O$ is $T_1$. 
\end{enumerate}

In particular, $\psi$ restricts to a bijection between $\Fix\varphi_G$ and the acyclic orientations of $G$ with a unique sink. 
\end{theorem}

\begin{proof}
By construction, because we orient edges from lower entries to higher entries, $O$ is indeed an acyclic orientation of $G$, inversions of $T$ are exactly ascents of $O$, and $T_1$ is a sink of $O$. By Lemma \ref{lem:legtoarm} (a), an entry $j$ in the leg of $T$ is movable if and only if it has no neighbours above it, which is exactly the condition for $O$ to have no outgoing edges from $j$ and for $j$ to be a sink. Also by Lemma \ref{lem:legtoarm}, every movable entry $T$ is larger than the entries above it, so the smallest sink of $O$ is $T_1$.

It remains to show that the map $\psi$ is a bijection. Given an acyclic orientation $O$, we construct the tableau $T=\psi^{-1}(O)$ as follows. We let $T_1$ be the smallest sink $j$ of $O$ and  delete vertex $j$ to obtain an acyclic orientation $O'$. Then we iterate this process, letting $T_2$ be the smallest sink of $O'$, and so on. The tableau $T$ satisfies (C2) because if $T_i>T_{i+1}$, the smaller entry $T_{i+1}$ became a sink only after vertex $T_i$ was removed, so $T_{i+1}T_i\in E$ and $T$ is indeed a $G$-tableau. 
The proof that the compositions $\psi\circ \psi^{-1}$ and $\psi^{-1}\circ\psi$ are identity maps is routine and so omitted.
\end{proof}

\section{Further directions}\label{sec:further}

We conclude by proposing some further avenues of study. It may be fruitful to explore the expansion of $X_G(\bm x)$ in other symmetric function bases. Tom used the power sum basis to produce a signed $e$-expansion of $X_G(\bm x)$ \cite[Theorem 5.10]{qforesttriples} and of $X_G(\bm x;q)$ whenever $G$ is a natural unit interval graph \cite[Theorem 3.4]{qforesttriples}. Cho and van Willigenburg found a way to generate many new bases for symmetric functions.

\begin{theorem}
\cite[Theorem 5]{chrombases} Fix a sequence of connected graphs $G_1,G_2,\ldots$ such that $G_k$ has $k$ vertices for every $k$. For a partition $\lambda=(\lambda_1,\ldots,\lambda_\ell)$, let $G_\lambda$ be the disjoint union of graphs $G_{\lambda_1}\cdots G_{\lambda_\ell}$.
Then the set $\{X_{G_\lambda}(\bm x): \lambda\vdash n\}$ is a basis for the space of symmetric functions of degree $n$. 
\end{theorem}

For example, we could take $G_k$ to be the \emph{star graph} $S_k$, which is the tree with $k$ vertices that has a vertex of degree $(k-1)$. Aliste-Prieto, de Mier, Orellana, and Zamora \cite[Theorem 3.2]{markedgraphs} found a combinatorial formula for the expansion of $X_G(\bm x)$ in the star basis $\{X_{S_\lambda}(\bm x)\}$. We could try using a change-of-basis between stars and elementary symmetric functions to further study $X_G(\bm x)$.

One can combine the Schur expansion of $X_G(\bm x;q)$~\eqref{eq:s} with the dual Jacobi-Trudi determinant~\eqref{eq:jt} to obtain a signed expansion into elementary symmetric functions for any natural unit interval graph.  The terms in this expansion  count pairs $(T,\pi)$ where $T$ is a $G$-tableau and $\pi$ is the permutation indicating the term in the determinant which was used, with the usual sign of a permutation as weight.  One could then try to prove Conjecture~\ref{SWconj} by finding a sign-reversing, inv-preserving involution on these pairs which could involve a new notion of movable element.  Cho and Huh~\cite{chromstoe} have used this approach to show $e$-positivity  of $X_G(\bm x;q)$ for natural unit interval graphs with $\alpha(G)= 2$.  Also, Cho and Hong~\cite{chrombounce3} showed $e$-positivity of $X_G(\bm x)$
when $\alpha(G)=3$.  What other classes of graphs can be handled in this manner?

When $G$ is a natural unit interval graph, Shareshian and Wachs conjectured that $X_G(\bm x;q)$ is not only $e$-positive but also \emph{$e$-unimodal}, meaning that for every $\lambda$ the coefficients of the polynomial $c_\lambda(q)=a_iq^i+\cdots+a_jq^j$ satisfy 
$$a_i\leq a_{i+1}\leq\cdots\leq a_{k-1}\leq a_k\geq a_{k+1}\geq\cdots\geq a_{j-1}\geq a_j$$
for some $i\leq k\leq j$. We could ask whether $X_G(\bm x;q)$ is \emph{$e$-log-concave}, meaning that every polynomial $c_\lambda(q)$ satisfies the stronger property that $a_\ell^2\geq a_{\ell+1}a_{\ell-1}$ for every $\ell$.   We have checked by computer that this holds for every natural unit interval graph with at most $10$ vertices.

\begin{conjecture}
Let $G$ be a natural unit interval graph. Then $X_G(\bm x;q)$ is $e$-log-concave.
\end{conjecture}

\begin{example}
For the family of graphs $K_{a,b}$ from Example \ref{ex:kab}, we saw in \eqref{eq:kab} that every coefficient in the $e$-expansion of $X_{K_{a,b}}(\bm x;q)$ is a product of $q$-integers.  So $X_{K_{a,b}}(\bm x;q)$ is $e$-log-concave.
\end{example}

\section{Acknowledgments}
The authors would like to thank Zachary Hamaker, Richard Stanley, and Vincent Vatter for helpful discussions.

\printbibliography

\end{document}